\documentclass[preprint,10pt]{elsarticle}
\usepackage{amsthm,amsmath,amssymb,mathrsfs}
\usepackage[colorlinks=true,citecolor=black,linkcolor=black,urlcolor=blue]{hyperref}
\usepackage{graphicx}
\usepackage{float}
\usepackage{epstopdf}
\usepackage{epsfig}
\usepackage{extarrows,chngpage,array,float,eqnarray}     

\journal{}
\theoremstyle{plain}
\newtheorem{theorem}{Theorem}[section]

\newtheorem{corollary}[theorem]{Corollary}

\theoremstyle{definition}
\newtheorem{definition}[theorem]{Definition}

\newtheorem{proposition}[theorem]{Proposition}

\theoremstyle{remark}
\newtheorem{remark}[theorem]{Remark}

\allowdisplaybreaks

\title{Connectivity for slice-projections of connected polymatroids}
\author{Xiaxia Guan$^{a}$,~~Xian'an Jin$^{b,c}$ \\
\small $^a$ Department of Mathematics, Taiyuan University of Technology, P. R. China\\
\small $^b$ School of Mathematical Sciences, Xiamen University, P. R. China\\
\small $^c$ School of Mathematics and Statistics, Qinghai Minzu University, P. R. China\\
\small \emph{Email addresses}: guanxiaxia@tyut.edu.cn; xajin@xmu.edu.cn}

\begin{document}
\begin{abstract}
It is well-known that deleting or contracting any element of a connected matroid always yields at least one connected minor. However, for a connected polymatroid, only two such elements can be guaranteed, proved by Hall in 2013. This note investigates the connectivity properties of slice-projections of connected polymatroids, which includes deletion and contraction. We establish that for any element of a connected polymatroid, at least one of its two consecutive slice-projections is connected. We also obtain that the $j$-th slice-projection from the top of a polymatroid and  $j$-th slice-projection from the bottom of its dual have the same connectedness. These results both extend existing connectedness theorems for graphs and matroids.
\end{abstract}

\begin{keyword}
Connectedness \sep Polymatroid\sep Slice-projection\sep Dual.
\MSC 05B35
\end{keyword}

\maketitle
\section{Introduction}
\noindent

Polymatroids constitute a profound generalization of matroids. They play an indispensable role in combinatorial optimization, matroid theory, and algebraic combinatorics.  Connectivity, as a fundamental structural invariant, quantifies the indecomposability of a combinatorial object, while connectivity-preserving operations lie at the heart of structural decomposition and minor theory.

A classical theorem of Tutte \cite{Tutte} asserts that for any element $e$ of a connected matroid $M$, at least one of $M-e$ or $M/e$ remains connected. Remarkably, this result fails to extend to the setting of polymatroids. In 2013, Hall \cite{Hall} proved that every connected polymatroid having at least  two elements has at least two elements $e$ such that either $P-e$ or $P/e$ is connected, and this bound is sharp. For further details on connectivity, refer to the literature \cite{Bonin,Oxley}. The purpose of this note is to investigate the connectivity properties of slice-projections of connected polymatroids, which includes deletion and contraction, and extend existing connectivity theorems for graphs and matroids. We first recall some necessary definitions.

\begin{definition}\label{polymatroid}
Let $E$ be a finite set (without loss of generality, we usually assume that $E=[n]=\{1,2,\ldots,n\}$) and $f:2^{E}\rightarrow \mathbb{Z}_{\geq 0}$ be a function. If $f$ satisfies the following three conditions:

\begin{enumerate}
\item[(i)] $f(\emptyset)=0$ (normalization);
\item[(ii)] $f(I)\leq f(J)$ for any $I\subseteq J\subseteq E$ (monotonicity);
\item[(iii)] $f(I)+f(J)\geq f(I\cup J)+f(I\cap J)$ for any $I,J\subseteq E$ (submodularity),
\end{enumerate}
then $f$ is called a \emph{rank function}. A \emph{polymatroid} $P\subseteq \mathbb{Z}^{n}_{\geq 0}$ on the ground set $E$ with rank function $f$ is given by
$$P=\left\{(a_{1},\ldots,a_{n})\in \mathbb{Z}^{n}_{\geq 0}~\middle|~\sum_{i\in I}a_{i}\leq f(I) \ \text{for all}\ I\subseteq E\right\}$$
and
$$\mathcal{B}_{P}=\left\{\textbf{a}=(a_{1},\ldots,a_{n})\in \mathbb{Z}^{n}_{\geq 0}~\middle|~\textbf{a}\in P \ \text{and} \ \sum_{i\in E}a_{i}=f(E)\right\}.$$
is the set of its all \emph{bases}.
\end{definition}

All independent sets of any matroid $M$ yields a polymatroid.
 More precisely, let $M\subseteq 2^{E}$ be a matroid. For any $B\subseteq E$, let
 $$I_{B}=(a_{1},\ldots,a_{n})\in \{0,1\}^{n}\ \text{where}\ a_{i}=1 \ \text{if and only if}\  i\in B.$$ Then
$$P=P(M)=\{I_{B} ~|~\text{$B$ is an independent set of} \ M\}\subseteq \{0,1\}^{n}\subseteq \mathbb{Z}^{n}_{\geq 0}$$ is the polymatoid corresponding to $M$.
\begin{definition}\label{del-con}
Let $P$ be a polymatroid on $E$ with rank function $f$. For an element $e\in E$, the \emph{deletion} $P-e$ and \emph{contraction} $P/e$, which are polymatroids on $E-e$, are given by the rank functions $f_{P-e}(T)=f(T)$ and $f_{P/e}(T)=f(T\cup e)-f(e)$, for any subset $T\subseteq E-e$, respectively.
\end{definition}

\begin{remark}
Deletion and contraction in Definition \ref{del-con} is the same as in the setting of matroids
\end{remark}

Following Mat\'{u}\v{s} \cite{Matus}, we define connected and disconnected polymatroids as follows.

\begin{definition}\label{def-co}
A polymatroid $P$ on $E$ with rank function $f$ is \emph{connected} or \emph{2-connected} if $f(I)+f(E-I) > f(E)$ for all nonempty proper subsets $I$ of $E$; otherwise, $P$ is \emph{disconnected}.
If $f(I)+f(E-I)=f(E)$, then $I$ is a \emph{separator}; it is nontrivial if $I \notin \{\emptyset, E\}$. When $I$ is a nontrivial separator, $(I,E-I)$ is called a \emph{1-separation} of $P$.
\end{definition}

\begin{remark}
Connectedness in Definition \ref{def-co} is also the same as in the setting of matroids.
\end{remark}

In \cite{Guan4}, the authors introduced slice-projections of a polymatriod and obtained its basic properties. Let $P$ be a polymatroid on $E$ with rank function $f$. For any $e\in E$, let
$$\alpha_{e}=f(E)-f(E-e),~~\beta_{e}=f(\{e\}),$$  and $$T_{e}=\{\alpha_{e}, \alpha_{e}+1,\ldots,\beta_{e}\}.$$
For any $j\in T_{e}$, let
 $$B_{P^{e}_{j}}:=\{(a_{1},\ldots,a_{n})\in B_{P}\mid a_{e}=j\},$$ and
\begin{equation*}\label{projection}
B_{\widehat{P}^{e}_{j}}:=\{(a_{1},\ldots,a_{e-1},a_{e+1},\ldots,a_{n})\in \mathbb{Z}^{n-1}\mid (a_{1},\ldots,a_{n})\in B_{P^{e}_{j}}\}.
\end{equation*}
In \cite{Guan4}, they proved that $B_{\widehat{P}^{e}_{j}}$ is the set of bases of a polymatroid on $E-e$.
\begin{definition}
Denote by $\widehat{P}^{e}_{j}$ the polymatroid whose set of bases is $B_{\widehat{P}^{e}_{j}}$, called a \emph{slice-projection} of $P$.
\end{definition}
Moreover, in \cite{Guan4}, they proved that the bottom and the top slice-projections coincide with the deletion and contraction, respectively, i.e.,
 \begin{eqnarray}\label{Del-Con}
 \widehat{P}^{e}_{\alpha_{e}}=P\setminus e\ \text{and} \ \widehat{P}^{e}_{\beta_{e}}=P/e.
\end{eqnarray}

Let $P$ be a polymatroid on $E$ with rank function $f$. If $|T_{e}|=1$ for some $e\in E$, then $(e,E-e)$ is a 1-separation of $P$. Hence,  $|T_{e}|\geq 2$ for any $e\in E$ if $P$ is connected. In this paper, we first proved:

\begin{theorem}\label{I-j}
Let $P$ be a connected polymatroid on $E$ with rank function $f$. For any $e\in E$ and any $j\in T_{e}$ satisfying $j\geq \alpha_{e}+1$, at least one of $\widehat{P}^{e}_{j}$ and $\widehat{P}^{e}_{j-1}$ is connected.
\end{theorem}

Note that if $M$ is a connected matroid on $E$ with  rank function  $f$, then $f(e)=1$ and $f(E)=f(E-e)$  for any $e\in E$, and hence $|T_{e}|=2$. By Equality (\ref{Del-Con}), the following corollary is immediate.
\begin{corollary}\cite{Tutte}
Let $M$ be a connected matroid on $E$.  Then at least one of $M-e$ and $M/e$ is connected for any $e\in E$.
\end{corollary}

Let $P$ be a polymatroid on $E$ with rank function $f$. Define the function $f^*$ as follows:
\begin{equation}\label{dual-rank}
f^*(I) = f(E-I) + \sum_{i \in I} f(i) - f(E) ~\text{for any}~I\subseteq E.
\end{equation}
In \cite{Jowett}, Jowett, Mo and Whittle proved that $f^*$ also satisfies the three conditions of rank function in Definition \ref{polymatroid}.
\begin{definition}
Denote $P^*$ the polymatroid given by  $f^*$, called the \emph{dual polymatroid} of $P$.
\end{definition}
Clearly, $P$ is connected if and only if $P^*$ is connected;
$f^*(E) = \sum_{i \in E} f(i) - f(E)$;
and for any  $e\in E$,
\begin{equation*}
f^*(e) = f(E-e) + f(e) - f(E),
\end{equation*}
\begin{eqnarray*}
f^*(E-e)&=&f(e)+\sum_{i \in E-e} f(i) - f(E)\\
&=& \sum_{i \in E} f(i) - f(E) \\
 &=& f^*(E).
 \end{eqnarray*}

This implies that
$\alpha^*_{e}=f^*(E)-f^*(E-e)=0$ and $\beta^*_{e}=f^*(e)=\beta_{e}- \alpha_{e}$. Define
$$T^*_{e}=\{\alpha^*_{e}, \alpha^*_{e}+1,\ldots,\beta^*_{e}\}.$$
Hence, $|T^*_{e}|=|T_{e}|=\beta_{e}- \alpha_{e}+1$. A polymatroid and its dual have the same number of slice-projections. We then obtain a slice-projection-connectedness duality theorem:
\begin{theorem}\label{dual-connected}
Let $P$ be a polymatroid on $E$ with rank function $f$. Then, for any $e\in E$ and any $j\in T^*_{e}$, the polymatroid $(\widehat{P^{*}})_{j}^{e}$ is connected if and only if $\widehat{P}^{e}_{f(e)-j}$ is connected.
\end{theorem}

Let $j=\alpha^*_{e}$ and $j=\beta^*_{e}$ in Theorem \ref{dual-connected}, respectively. Then
the following results are immediate consequences of Theorem \ref{dual-connected} by Equality (\ref{Del-Con}).
 \begin{corollary}\label{dual-equal}
Let $P$ be a polymatroid on $E$ with dual polymatroid $P^*$.  Then for any $e\in E$,
\begin{enumerate}
\item [(i)] $P/e$ and $P^*\setminus e$ are equally connected;
\item [(ii)] $P\setminus e$ and $P^*/e$ are equally connected.
\end{enumerate}
\end{corollary}

\begin{remark}
$P$ and $P^*$ are equally connected. However, it is possible that $(P/e)^*\neq P^*\setminus e$ and $(P\setminus e)^*\neq P^*/e$. Hence, $(i)\nLeftrightarrow(ii)$ in Corollary \ref{dual-equal}.
\end{remark}

 \section{Proof of main results}
\noindent

In this section, we prove Theorems \ref{I-j} and \ref{dual-connected}, respectively. The following proposition in \cite{Guan4} will be used.

\begin{proposition}\label{rank function of Pj}\cite{Guan4}
Let $P$ be a polymatroid on $E$ with rank function $f$. For any $e\in E$ and $j\in T_{e}$, let $f^{e}_{j}$ be the rank function of the polymatroid $\widehat{P}^{e}_{j}$. Then for any subset $I\subseteq E-e$, we have $$f^{e}_{j}(I)=\min\{f(I), f(I\cup e)-j\}.$$
\end{proposition}

 \subsection{Proof of Theorem \ref{I-j}}
\noindent

By contradiction, assume that neither  $\widehat{P}^{e}_{j-1}$ nor $\widehat{P}^{e}_{j}$ is  connected for some $j\geq \alpha_{e}+1$. Then  by Definition \ref{def-co}, there exists 1-separation $(I,E-e-I)$ and $(J,E-e-J)$  in polymatroids $\widehat{P}^{e}_{j}$ and $\widehat{P}^{e}_{j-1}$, respectively,
that is, $$f_{j}^e(I) + f_{j}^e(E-e-I)= f_{j}^e(E-e)$$ and
 $$f_{j-1}^e(J) + f_{j-1}^e(E-e-J) = f_{j-1}^e(E-e).$$

Note that $j\geq \alpha_{e}+1$. Then, by Proposition \ref{rank function of Pj}, we have that
\begin{equation*}\label{I}
f_{j}^e(E-e) = f(E)-j ~~\text{and}~~f_{j-1}^e(E-e) = f(E)-(j-1).
\end{equation*}

Clearly, $I\neq \emptyset$, $I\neq E-e$, $J\neq \emptyset$ and $J\neq E-e$. So, it is impossible that both \(f_{j}^e(I)= f(I\cup e)-j\) and \(f_{j}^e(E-e-I)=f(E-e-I)\) hold simultaneously, and it is also impossible that both  \(f_{j}^e(I)= f(I)\) and \(f_{j}^e(E-e-I)= f(E-I)-j\) hold simultaneously. Otherwise,
\begin{equation*}
f(I\cup e) + f(E-e-I) = f(E),
\end{equation*}
or
\begin{equation*}
f(I) + f(E-I) = f(E).
\end{equation*}
This implies that $(I \cup e,E-(I \cup  e))$ or $(I ,E-I)$ is a 1-separation of $P$, which  contradicts the fact that $P$ is connected. Similarly, it is impossible that both \(f_{j-1}^e(J)= f(J\cup e)-(j-1)\) and \(f_{j-1}^e(E-e-J)=f(E-e-J)\) hold simultaneously, and it is also impossible that both   \(f_{j-1}^e(J)= f(J)\) and \(f_{j-1}^e(E-e-J)= f(E-J)-(j-1)\) hold simultaneously. This implies that the following claim holds.

\vskip0.2cm
\noindent\textbf{Claim 1.} \(f(I)\neq f(I\cup e)-j\), \(f(E-e-I) \neq f(E-I)-j\), \(f(J) \neq f(J\cup e)-(j-1)\) and \(f(E-e-J) \neq f(E-J)-(j-1)\).
\vskip0.2cm
\noindent\emph{Proof of Claim 1.}
Assume that \(f(I)= f(I\cup e)-j\). Then \(f_{j}^e(I)=f(I)= f(I\cup e)-j\).  If \(f_{j}^e(E-e-I)=f(E-e-I)\), then  both \(f_{j}^e(I)= f(I\cup e)-j\) and \(f_{j}^e(E-e-I)=f(E-e-I)\) hold simultaneously, a contradiction. If \(f_{j}^e(E-e-I)= f(E-I)-j\), then both  \(f_{j}^e(I)= f(I)\) and \(f_{j}^e(E-e-I)= f(E-I)-j\) hold simultaneously, a contradiction. Hence, \(f(I)\neq f(I\cup e)-j\). Similarly, one can show that \(f(E-e-I) \neq f(E-I)-j\), \(f(J) \neq f(J\cup e)-(j-1)\) and \(f(E-e-J) \neq f(E-J)-(j-1)\). We completes the proof of Claim 1.
 \vskip0.2cm

By Claim 1, it suffices to consider the following four cases.

\noindent\textbf{Case 1.} Suppose that \(f(I)> f(I\cup e)-j\), \(f(E-e-I) > f(E-I)-j\), \(f(J) < f(J\cup e)-(j-1)\) and \(f(E-e-J) < f(E-J)-(j-1)\).

\vskip0.2cm
Then \begin{equation}\label{I2}
 f(J\cup e)-f(J)>j-1\geq f(E-I)-f(E-e-I),
 \end{equation}
 and
 \begin{equation}\label{I1}
f(E-J)-f(E-e-J)>j-1\geq  f(I\cup e)-f(I).
\end{equation}
Moreover,
by Proposition \ref{rank function of Pj}, we have that
\begin{equation}\label{I}
f(I\cup e) + f(E-I) = f(E)+j,
\end{equation}
and
\begin{equation}\label{J-1}
f(J) + f(E-e-J) = f(E)-(j-1).
\end{equation}
By the submodularity of $f$, we have that
\begin{equation}\label{IJe}
f(I\cup e)+f(J) \geq f(I \cup J\cup e) + f(I \cap J),
\end{equation}
and
\begin{equation}\label{IJ}
f(E-I) + f(E-e-J) \geq f(E-(I \cup J\cup e)) + f(E-(I \cap J)).
\end{equation}

We now prove the following claim.
\vskip0.2cm
\noindent\textbf{Claim 2.} $I \cup J\cup e\neq E$ and $I \cap J\neq \emptyset$.
\vskip0.2cm
\noindent\emph{Proof of Claim 2.}
If $I \cup J\cup e=E$, then $ E-I\subseteq J\cup e$. By the submodularity of $f$, we have that $ f(E-I)+f(J)\geq f(E-e-I)+f(J\cup e)$, which contradicts Inequality (\ref{I2}). So, $I \cup J\cup e\neq E$. If $I \cap J=\emptyset$, then $I\cup e\subseteq E-J$. By the submodularity of $f$, we have that  $f(I\cup e)+f(E-e-J)\geq f(E-J)+f(I)$, which is in contradiction to Inequality (\ref{I1}). Thus,  $I \cap J\neq \emptyset$. We completes the proof of Claim 2.
 \vskip0.2cm
Claim 2 yields that both $(I \cup J\cup e,E-(I \cup J\cup e))$ and $(I \cap J,E-(I \cap J))$ are partitions of $E$. Hence, by Definition \ref{def-co}, we have
\begin{equation}\label{IJe1}
f(I \cup J\cup e) + f(E-(I \cup J\cup e)) \geq f(E)+1,
\end{equation}
and
\begin{equation}\label{IJ1}
f(I \cap J) + f(E-(I \cap J)) \geq f(E)+1.
\end{equation}
However, adding Inequalities (\ref{IJe}), (\ref{IJ}), (\ref{IJe1}) and (\ref{IJ1}), we have
\begin{equation*}
f(I\cup e) + f(J)+ f(E-I) + f(E-e-J)\geq 2f(E)+2.
\end{equation*}
Adding Equalities (\ref{I}) and (\ref{J-1}),
\begin{equation*}
f(I\cup e) + f(J)+ f(E-I) + f(E-e-J)=2f(E)+1,
\end{equation*}
a contradiction.
 \vskip0.2cm
\noindent\textbf{Case 2.}  Suppose that \(f(I)> f(I\cup e)-j\), \(f(E-e-I) > f(E-I)-j\), \(f(J) > f(J\cup e)-(j-1)\) and \(f(E-e-J) > f(E-J)-(j-1)\).
\vskip0.2cm
By Proposition \ref{rank function of Pj}, we have Equality (\ref{I}) and
\begin{equation}\label{J-3}
f(J \cup e) + f(E-J) = f(E)+j-1.
\end{equation}
By adding Equalities (\ref{I}) and (\ref{J-3}), and the submodularity of $f$, we have that
\begin{eqnarray*}
&&2f(E)+2j-1\\
&=& f(I \cup e) + f(E-I) +f(J \cup e) + f(E-J)\\
&=& (f(I \cup e) +f(J \cup e))+ (f(E-I) + f(E-J))\\
&\ge &(f(I \cup J \cup e) + f((I \cap J) \cup e)) + (f(E-(I \cup J)) + f(E-(I \cap J))) \\
&= &(f(I \cup J \cup e) + f(E-(I \cup J)))+(f((I \cap J) \cup e) + f(E-(I \cap J))) \\
&\ge &2f(E) + 2f(e).
\end{eqnarray*}
This implies that $j\ge \frac{1}{2}+f(e)$,  which  contradicts  $j\leq f(e)$.

 \vskip0.2cm
\noindent\textbf{Case 3.} Assume that \(f(I) < f(I\cup e)-j\), \(f(E-e-I) < f(E-I)-j\), \(f(J) < f(J\cup e)-(j-1)\) and \(f(E-e-J) < f(E-J)-(j-1)\).
\vskip0.2cm
By Proposition \ref{rank function of Pj}, we have that Equality (\ref{J-1}) and
\begin{equation}\label{I-1}
f(I) + f(E-e-I) = f(E)-j.
\end{equation}
By adding Equalities (\ref{J-1}) and (\ref{I-1}), the submodularity of $f$ and the value of $j$, we have
\begin{eqnarray*}
&&2f(E)-2j+1\\
&=&(f(I) + f(J))+ (f(E-e-I) + f(E-e-J))\\
&\ge &(f(I \cup J) + f(I \cap J))+ (f(E-e-(I \cup J)) + f(E-e-(I \cap J))) \\
&= &(f(I \cup J) +f(E-e-(I \cup J)))+ (f(I \cap J) + f(E-e-(I \cap J))) \\
 &\ge& 2f(E-e)\\
 &\ge& 2f(E)-2j+2,
 \end{eqnarray*}
a contradiction.

 \vskip0.2cm
\noindent\textbf{Case 4.} Assume that \(f(I) < f(I\cup e)-j\), \(f(E-e-I) < f(E-I)-j\), \(f(J) >f(J\cup e)-(j-1)\) and \(f(E-e-J) > f(E-J)-(j-1)\).
\vskip0.2cm
Then  Equalities (\ref{J-3})  and (\ref{I-1}) hold.
 By adding Equalities (\ref{J-3}) and (\ref{I-1}), the submodularity of $f$, we have
\begin{eqnarray*}
&&2f(E)-1\\
&=&(f(I) + f(J\cup e))+ (f(E-e-I) + f(E-J))\\
&\ge& (f(I \cup J\cup e) + f(I \cap J))+ (f(E-(e\cup I \cup J)) + f(E-(I \cap J))) \\
&\ge& (f(I \cup J\cup e) + f(E-(e\cup I \cup J)))+ (f(I \cap J)+ f(E-(I \cap J))) \\
 &\ge& 2f(E),
 \end{eqnarray*}
a contradiction.

 \subsection{Proof of Theorem \ref{dual-connected}}
\noindent
We begin with the following claim.
\vskip0.2cm
\noindent\textbf{Claim 1.} For any $j\in T^*_{e}$ and for any subset $I\subseteq E-e$, we have that \(f(I) > f(I\cup e)-j\) if and only if \(f(E-I-e)<f(E-I) -(f(e) -j)\).
\vskip0.1cm
\noindent\emph{Proof of Claim 1.}  By the definition of $f^*$ (see Equality (\ref{dual-rank})), we have
 \begin{eqnarray*}
&& f(I) > f(I\cup e)-j\\
 & \Leftrightarrow &f(E-I) + \sum_{i \in I} f(i) - f(E)>f(E-(I\cup e)) + \sum_{i \in I\cup e} f(i) - f(E)-j\\
  & \Leftrightarrow &f(E-I)>f(E-I-e)+(f(e) -j)\\
 & \Leftrightarrow &f(E-I-e)<f(E-I) -(f(e) -j).
  \end{eqnarray*}

To prove Theorem \ref{dual-connected}, it suffices to prove the following claim.
\vskip0.2cm
\noindent\textbf{Claim 2.} For any subset $I\subseteq E$, we have that $(I,E-e-I)$ is a 1-separation of the polymatroid $(\widehat{P^*})_{j}^{e}$ if and only if $(I,E-e-I)$ is also a 1-separation of the polymatroid $\widehat{P}^{e}_{f(e)-j}$.
\vskip0.1cm
\noindent\emph{Proof of Claim 2.} Let $f_{f(e)-j}^{e}$ and $(f^*)_{j}^{e}$ be the rank functions of polymatroids $\widehat{P}_{f(e)-j}^{e}$ and $(\widehat{P^*})_{j}^{e}$, respectively. By the definition of 1-separations, Proposition \ref{rank function of Pj} and the submodularity of $f$, we only need to prove that
 $$(f^*)_{j}^{e}(I) + (f^*)_{j}^{e}(E-e-I)= \sum_{i \in E} f(i) - f(E)-j$$ if and only if $$f_{f(e)-j}^{e}(I) + f_{f(e)-j}^{e}(E-e-I)= f(E)-(f(e)-j).$$
Note that the case \(f^*(I) > f^*(I\cup e)-j\) and \(f^*(E-e-I) \leq f^*(E-I)-j\), and the case \(f^*(I) \leq f^*(I\cup e)-j\) and \(f^*(E-e-I) > f^*(E-I)-j\) are symmetrical. Hence, we divide the proof of Claim 2 into three cases.
\vskip0.2cm
\noindent \textbf{Case 1.} Assume that \(f^*(I) \leq f^*(I\cup e)-j\) and \(f^*(E-e-I) \leq f^*(E-I)-j\). Then by Claim 1,
$$f(E-I-e)\geq f(E-I) -(f(e) -j) ~\text{and} ~f(I) \geq f(I\cup e)-(f(e) -j).~~~~(2-1)$$
 By Proposition \ref{rank function of Pj},
 \begin{eqnarray*}
&&(f^*)_{j}^{e}(I) + (f^*)_{j}^{e}(E-e-I)= \sum_{i \in E} f(i) - f(E)-j \\
& \Leftrightarrow & f^*(I) + f^*(E-e-I)= \sum_{i \in E} f(i) - f(E)-j. ~~~~~~~~~~~~~~~~~~~~~~~~~(2-2)
  \end{eqnarray*}
 Next by the definition of $f^*$, we have that
 \begin{eqnarray*}
(2-2) & \Leftrightarrow &\left(f(E-I) + \sum_{i \in I} f(i) - f(E)\right)  +\left( f(I\cup e) + \sum_{i \in E-e-I} f(i) - f(E)\right)\\
 & &=\sum_{i \in E} f(i) - f(E)-j\\
 & \Leftrightarrow &f(E-I) + f(I\cup e) - f(E)=  f(e)-j\\
 & \Leftrightarrow &f(E-I)-(f(e)-j) + f(I\cup e) -(f(e)-j)\\
 & & =f(E)-(f(e)-j).     ~~~~~~~~~~~~~~~~~~~~~~~~~~~~~~~~~~~~~~~~~~~~~~~~~~(2-3)
  \end{eqnarray*}
So, by  $(2-1)$ and Proposition \ref{rank function of Pj}, we have
\begin{eqnarray*}
(2-3) & \Leftrightarrow &f_{f(e)-j}^{e}(E-e-I) +f_{f(e)-j}^{e}(I)= f(E)-(f(e)-j).
  \end{eqnarray*}
\vskip0.2cm
\noindent \textbf{Case 2.} Assume that \(f^*(I) > f^*(I\cup e)-j\) and \(f^*(E-e-I) \leq f^*(E-I)-j\). Then by the definition of $f^*$,  Proposition \ref{rank function of Pj} and Claim 1, we have
\begin{eqnarray*}
&&(f^*)_{j}^{e}(I) + (f^*)_{j}^{e}(E-e-I)= \sum_{i \in E} f(i) - f(E)-j \\
& \Leftrightarrow & f^*(I\cup e)-j + f^*(E-e-I)= \sum_{i \in E} f(i) - f(E)-j~~~\text{(Proposition \ref{rank function of Pj})}\\
 & \Leftrightarrow &\left(f(E-I-e)+ \sum_{i \in I\cup e} f(i) - f(E)\right) -j + \left(f(I\cup e) + \sum_{i \in E-e-I} f(i) - f(E)\right)\\
 & &=\sum_{i \in E} f(i) - f(E)-j~~~~~~~~~~~~~~~~~~~~~~~~~~~~~~~~~~~~~~\text{(the definition of $f^*$)}\\
 & \Leftrightarrow &f(E-I-e) + f(I\cup e) - f(E)=  0\\
 & \Leftrightarrow &f(E-I-e) + f(I\cup e) -(f(e)-j)= f(E)-(f(e)-j)\\
 & \Leftrightarrow & f_{f(e)-j}^{e}(E-e-I)+f_{f(e)-j}^{e}(I)= f(E)-(f(e)-j).\\
& & ~~~~~~~~~~~~~~~~~~~~~~~~~~~~~~~~~~~~~~~~~~~~~~~~~~~~~~~\text{(Claim 1 and Proposition \ref{rank function of Pj})}
  \end{eqnarray*}
\vskip0.2cm
\noindent   \textbf{Case 3.} Assume that \(f^*(I) > f^*(I\cup e)-j\) and \(f^*(E-e-I) > f^*(E-I)-j\). Then by the definition of $f^*$,  Proposition \ref{rank function of Pj} and Claim 1, we have
\begin{eqnarray*}
&&(f^*)_{j}^{e}(I) + (f^*)_{j}^{e}(E-e-I)= \sum_{i \in E} f(i) - f(E)-j \\
& \Leftrightarrow &f^*(I\cup e)-j + f^*(E-I)-j= \sum_{i \in E} f(i) - f(E)-j~~~~\text{(Proposition \ref{rank function of Pj})}\\
 & \Leftrightarrow &\left(f(E-I-e)+ \sum_{i \in I\cup e} f(i) - f(E)\right)-j  + \left(f(I) + \sum_{i \in E-I} f(i) - f(E)\right)\\
 & &-j=\sum_{i \in E} f(i) - f(E)-j~~~~~~~~~~~~~~~~~~~~~~~~~~~~~~~~~~\text{(the definition of $f^*$)}\\
 & \Leftrightarrow &f(E-I-e) + f(e)+f(I) - f(E)-j = 0\\
 & \Leftrightarrow &f(E-I-e)+f(I)=f(E)-(f(e)-j)\\
 & \Leftrightarrow &f_{f(e)-j}^{e}(E-e-I)+f_{f(e)-j}^{e}(I)=f(E)-(f(e)-j).\\
& & ~~~~~~~~~~~~~~~~~~~~~~~~~~~~~~~~~~~~~~~~~~~~~~~~~~~~~~~\text{(Claim 1 and Proposition \ref{rank function of Pj})}
  \end{eqnarray*}

\section{Discussions}

We discuss the reductions of Theorem \ref{dual-connected} to matroids and hypergraphs. Let $M$ be a matroid on $E$ with rank function $r$. Recall that the rank function $r^*$
\begin{equation}\label{dual-rank1}
r^*(I) = r(E-I) + |I| - r(E) ~\text{for any}~I\subseteq E.
\end{equation}
 Denote $M^*$ the matroid given by  $r^*$, called the \emph{dual matroid} of $M$. Clearly, the connectivity properties of $M$ and $M^*$ are identical. Moreover,
\begin{equation*}
M^*-e = (M/e)^* \text{~and~} M^* /e = (M-e)^*.
\end{equation*}

Hence, we have the following results (see \cite{Oxley1}), as also follows from Corollary \ref{dual-equal}.
\begin{theorem}\label{dual-equal1}
For a matroid $M$ on $E$ with dual matroid $M^*$.  Then for any $e\in E$,
\begin{enumerate}
\item [(1)] $M/e$ and $M^*-e$ are equally connected;
\item [(2)] $M-e$ and $M^*/e$ are equally connected.
\end{enumerate}
\end{theorem}

\begin{proof}
Note that the connectedness of a matroid and its corresponding polymatroid are identical.
Let $P=P(M)$ and $P(M^*)$ be the polymatroids corresponding to $M$ and $M^*$, respectively. Then $P-e$ and $P/e$ are the polymatroids corresponding to $M-e$ and $M/e$, respectively.

Let $r$ and $r^*$ be the rank functions of matroids $M$ and $M^*$, respectively. If $r(e')\neq 0$ for any $e'\in E$, then $P^*=P(M^*)$ by Equalities (\ref{dual-rank}) and (\ref{dual-rank1}). This implies that $P^*-e$ and $P^*/e$ are the polymatroids corresponding to $M^*-e$ and $M^*/e$, respectively. By Corollary \ref{dual-equal}, the conclusions hold.

If $r(e')= 0$ for some $e'\in E-e$, then $r^*(E-e')=r^*(E)-1$. Hence, $M/e$, $M^*-e$,  $M-e$ and $M^*/e$ are disconnected. Let $f$ and $f^*$ be the rank functions of polymatroids $P$ and $P^*$, respectively. Clearly, $f(e')=f^*(e')=0$, so, $|T^*_{e'}|=|T_{e'}|=1$. Hence, $P/e$, $P^*-e$,  $P- e$ and $P^*/e$ are also disconnected. In this case, the conclusions hold.

We now assume that $r(e)= 0$ and $r(e')\neq 0$ for any $e'\in E-e$. Then 	
\begin{equation*}\label{Tutte-DC}
				f^*(I)=\left\{\begin{array}{ll}
					r^*(I),&\text{ if } e\notin I,\\
                    r^*(I)-1,&\text{ if } e\in I.
				\end{array}\right.
			\end{equation*}
This implies that $P^*/e=P(M^*)/e$ and $P^*- e=P(M^*)-e$ by the definitions of contraction and deletion. Hence, the conclusions also hold by Corollary \ref{dual-equal}.
\end{proof}

If $P$ is a hypergraphical polymatroid, then both $\widehat{P}^{t}_{\alpha_{t}}$ and $\widehat{P}^{t}_{\beta_{t}}$ are hypergraphical. However,  maybe $\widehat{P}^{t}_{j}$ is not hypergraphical for $j\in T_{t}\setminus \{\alpha_{t}, \beta_{t}\}$. For example, if $P$ is a polymatroid induced by a hypergraph $\mathcal{H}=(V,E)$, where $V=\{1,2,3,4,5,6\}$, $E=\{e=\{1,2,3\},a=\{1,4,5\},b=\{2,4,6\},c=\{3,5,6\}\}$, then $\widehat{P}^{e}_{1}$ is not hypergraphical.
So Theorem \ref{dual-connected} can not be reduced to the setting of hypergraphs.

\section*{Acknowledgements}
\noindent

This  was supported by National Natural Science Foundation of China (No. 12401462 and 12571379), the Natural Science Foundation of Shanxi Province (No. 202403021222034) and the Shanxi Key Laboratory of Digital Design and Manufacturing (No. 202204010931025).

\section*{Declarations}
\noindent
{ \bf Conflict of interest}
The authors declare that they have no conflict of interest.


%
%
\end{document}